\documentclass[12pt]{amsart}
\usepackage{hyperref}
\usepackage[utf8]{inputenc}
\usepackage{amssymb, graphicx}
\usepackage{amsfonts}
\usepackage{amsmath}
\usepackage{latexsym}
\usepackage{amscd}
\usepackage{verbatim}
\usepackage{mathrsfs}
\usepackage{color}

\allowdisplaybreaks[4]

\newcommand{\mf}{\mathfrak}
\newcommand{\g}{\mf{g}}
\newcommand{\h}{\mf{h}}

\allowdisplaybreaks[4]

\def\h{\mathfrak{h}}

\def\DD{\mathcal{L}}
\def\mh{\mathcal{H}}
\def\DD{\mathfrak{D}}

\def\bN{{\mathbb N}}
\def\bZ{{\mathbb Z}}

\def\bR{{\mathbb R}}

\def\bC{{\mathbb C}}

\def\de{\delta}

\newcommand{\Z}{{\mathbb Z}}
\newcommand{\C}{{\mathbb C}}
\newcommand{\N}{{\mathbb N}}

\newcommand{\supp}{{\operatorname{Supp}}\xspace}

\renewcommand{\phi}{\varphi}

\def\Ann{\mathrm{Ann}}
\def\Ind{\mathrm {Ind}}

\newtheorem{theorem}{Theorem}[section]

\newtheorem{lemma}[theorem]{Lemma}

\theoremstyle{remark}

\newtheorem{definition}[theorem]{Definition}

\numberwithin{equation}{section}
\def\Vir{\mathrm{Vir}}
\def\span{\mathrm{span}}
\def\supp{\mathrm{Supp}}

\def\vn{\varnothing}

\def\S{\mathcal S}

\def\1{\mathbf 1}

\begin{document}




            \title[]{Representations of the Fermion-Virasoro algebras}
            \author{Yaohui Xue, Kaiming Zhao}
            \maketitle

            \begin{abstract} Let $\delta=0$ or $\frac{1}{2}$. In this paper, we introduce  the Fermion algebra $F(\delta)$ and  the Fermion-Virasoro algebra $\S(\delta)$. They are infinite-dimensional Lie superalgebras. All simple smooth $F(\delta)$-modules,  all simple  weight $F(\delta)$-modules, all simple smooth $\S(\delta)$-modules of nonzero level, and all simple Harish-Chandra $\S(\delta)$-modules   are determined.   Surprisingly, we have found four different    $\S(\delta)$-module  structures on  free $\C[L_0]$-modules of rank $2$.

                        \vspace{0.3cm}


            \end{abstract}
            %

            %



            \section{Introduction}

            Throughout the paper we denote by $\bZ, \bZ^*, \bZ_+, \bN, \bR, \bC,$ and $\bC^*$ the sets of integers, nonzero integers, non-negative integers, positive integers, real numbers, complex numbers, and nonzero complex numbers, respectively. All vector spaces and algebras are assumed to be over $\bC$.
            For a Lie superalgebra $\g$, the universal enveloping algebra of $\g$ is denoted by $U(\g)$.

            In this paper we will introduce and study the Lie superalgebras $\S(0)$ and $\S(\frac{1}{2})$, which we consider  to be the super version of twisted Heisenberg-Virasoro algebra  $  \DD$ and are called Fermion-Virasoro algebras. We first recall the Lie algebra  $  \DD$ from \cite{ACKP}.

            \begin{definition}\label{def.1}
                        The  {\bf twisted Heisenberg-Virasoro algebra} ${ {\mathfrak{D}}}$ is a Lie algebra with a basis
                        $$\left\{d_{m},h_{r}, c_1, c_2,  {c}_3:m, r\in\bZ\right\}$$ and subject to the commutation relations
                        \begin{equation}\label{thva}\begin{aligned}
                        &[d_m, d_n]=(m-n)d_{m+n}+\de_{m+n,0}\frac{ m^3-m}{12} c_1,\\
                        &[d_m,h_{r}]=-rh_{m+r}+\delta_{m+r,0}(m^2+m) c_2,\\
                        &[h_{r},h_{s}]=r\de_{r+s,0} {c}_3,\\
                        &[ c_1,{ {\mathfrak{D}}}]=[ c_2,{ {\mathfrak{D}}}]=[ c_3,{ {\mathfrak{D}}}]=0,\end{aligned}
                        \end{equation}
                        for $m,n,r,s\in\bZ$.
            \end{definition}

            It is clear that ${ {\mathfrak{D}}}$ contains a copy of the Virasoro subalgebra $\Vir=\text{span}\{ c_1, d_i:  i\in\bZ\}$ and the Heisenberg algebra $  {\mh}=\bigoplus_{r\in\bZ}\bC h_r\oplus\bC  {c}_3$.

            Similarly, for $\delta=0$ or $\frac{1}{2}$, the Fermion-Virasoro algebra  $\S(\delta)$ contains a copy of the Virasoro subalgebra $\Vir$ and the Fermion algebra $F(\delta)$, see Definition 2.1.

            The representation theory over various infinite-dimensional Lie (super)algebras has been studied by many mathematicians and physicists recently. For example, Simple smooth modules for the Virasoro algebra, the Neveu-Schwarz algebra were determined in \cite{MZ, LPX}. Y. Billig and V. Futorny classified the simple Harish-Chandra modules over the Witt algebra of rank $n$ in \cite{BF}; H. Tan, Y. Yao and K. Zhao classified the simple smooth modules over the mirror and twisted Heisenberg-Virasoro algebras with nonzero level in \cite{TYZ}. For more related results, see the references within these papers.

            Let $\delta=0$ or $\frac{1}{2}$. In this paper, we classify all simple smooth $F(\delta)$-modules,  all simple  weight $F(\delta)$-modules, all simple smooth $\S(\delta)$-modules  of nonzero level, all simple Harish-Chandra $\S(\delta)$-modules, and all $\S(\delta)$-modules structures on  free $\C[L_0]$-modules of rank $2$ with non-trivial odd spaces.

            This paper is arranged as follows. In Section 2, the Fermion algebras and Fermion-Virasoro algebras are introduced, and related definitions are collected. In Section 3, we give some property for the representations of $F(\delta)$. In particular, we classify all simple smooth $F(\delta)$-modules in Lemma \ref{lemrF0} and Lemma \ref{rFm}, and classify
            simple  weight $F(\delta)$-modules in Lemma \ref{weightF}. Also we show that any smooth $F(\delta)$-module of nonzero level is completely reducible in Lemma \ref{cr2}. In Section 4, using vertex operators (infinite sum of fermionic oscillators) we classify all simple smooth $\S(\delta)$-modules  of nonzero level in Theorems \ref{smooth1} and \ref{smooth2}. In Section 5, using  properties of weight $F(\delta)$-modules we classify all simple Harish-Chandra $\S(\delta)$-modules in Theorems \ref{HC1} and   \ref{HC2}. In Section 6, we determine  all $\S(\delta)$-modules structures on  free $\C[L_0]$-modules of rank $2$ with non-trivial odd spaces. Surprisingly we have found four different classes of such modules in Theorems \ref{key}.

            At the end of  this paper we propose two questions which may be interesting to   readers.

            \section{The Fermion-Virasoro algebras}

            Throughout this paper, $\delta=0$ or $\frac{1}{2}$.
            Now we define the Fermion algebra $F(\delta)$ that was used implicitly in Section 3.4 to construct unitary highest weight representation of the Virasoro algebra in \cite{KRR}.

            \begin{definition}\label{def.2}
                        The  {\bf Fermion algebra} $F(\delta)=F(\delta)_{\bar 0}\oplus F(\delta)_{\bar 1}$ is a Lie superalgebra with a basis
                        $$\left\{\psi_{m},z:m\in\bZ+\delta\right\}$$ and subject to the commutation relations
                        \begin{equation}
                        [\psi_m,\psi_n]=\delta_{m,-n}z, \,\,[F(\delta),z]=0,
                        \end{equation}
                        for $m,n\in\bZ+\delta$, where $F(\delta)_{\bar 1}=\sum_{i\in\Z+\delta}\C\psi_i$ and $F(\delta)_{\bar 0}=\C z$.
            \end{definition}

Write $F:=F(0)$ for short.
We can consider the   Fermion algebra $F$ as the super version of the Heisenberg algebra $  {\mh}$. The operators $\psi_m$ are generally called fermionic oscillators.

            Now we define the Fermion-Virasoro algebra $\S(\delta)$   that was implicitly used in Section 3.4 in \cite{KRR} although the algebra $\S(\delta)$ was not defined there.

            \begin{definition}\label{def.3} The {\bf Fermion-Virasoro algebra}   $\S(\delta)=\S(\delta)_{\bar 0}\oplus \S(\delta)_{\bar 1}$ is the Lie superalgebra with basis $\{L_m,\psi_n,c,z:m\in\Z,n\in\Z+\delta\}$, where $L_m,c,z\in\S(\delta)_{\bar 0},\psi_n\in\S(\delta)_{\bar 1}$, and  subject to the commutation relations
                        $$\aligned {} [L_m,L_n]&=(m-n)L_{m+n}+\delta_{m,-n}\frac{m^3-m}{12}c,\\
                        [L_m,\psi_n]&=(-n-\frac{m}{2})\psi_{m+n},\\ [\psi_m,\psi_n]&=\delta_{m,-n}z,\\
                        [\S(\delta),c]&=[\S(\delta),z]=0.\endaligned$$
            \end{definition}

            Write $\S:=\S(0)$ for short. The results in this paper are proved for $\S$ and $F$ while their proofs are also applicable to those for $\S(\frac{1}{2})$ and $F(\frac{1}{2})$. The main results are given in both $\delta=0$ and $\frac{1}{2}$.

            The Lie superalgebra $\S(0)/\C z$ is isomorphic to a quotient algebra of a special case of the Lie superalgebras studied in \cite{CLW}. The Lie superalgebras $\S(\delta)/\C z$ were discussed in \cite{DCL}, where their simple Harish-Chandra modules were classified.

            It is clear that $\S$ contains a copy of the Virasoro subalgebra $\Vir=\text{span}\{c, L_i:  i\in\bZ\}$ and a  copy of  the  Fermion algebra $F=\bigoplus_{r\in\bZ}\bC \psi_r\oplus\bC z$.
            We know that  $\S$ has a $\Z$-gradation $\S=\oplus_{k\in\Z}\S_k$, where $$\S_k=\C L_k\oplus\C\psi_k,\forall k\in\Z^*;\,\,\,\,\S_0=\C L_0\oplus\C\psi_0\oplus\C c\oplus\C z.$$ Write $\S_-:=\oplus_{k\in-\N}\S_k,\S_+:=\oplus_{k\in\N}\S_k$. Then $\S$ has a triangular decomposition $$\S=\S_-\oplus\S_0\oplus\S_+.$$ Set $\Vir_\pm=\Vir\cap\S_\pm,F_\pm=F\cap\S_\pm$.

            For any $m\in\Z_+, n\in\Z$, define the following subalgebras of $\S$:
            $$\aligned 
            &\S^{(m,-\infty)}=\bigoplus_{i\in\Z_+}\C L_{m+i}\oplus\bigoplus_{j\in\Z}\C\psi_{j}\oplus\C c\oplus\C z,\\
            &F^{(n)}=\bigoplus_{i\in\Z_+}\C\psi_{n+i}\oplus\C z,\\
    &F_{[m]}=\bigoplus_{i=-m}^m\C\psi_i\oplus\C z.\endaligned$$

Let $\g=\S(\delta),F(\delta)$ or $F_{[m]},m\in\Z_+$.
For any $\mu\in\C^*$, there is an automorphism $\sigma_\mu$ of $\g$ given by
$$\sigma_\mu(L_k)=L_k,\sigma_\mu(\psi_l)=\mu\psi_l,\sigma_\mu(c)=c,\sigma_\mu(z)=\mu^2z,k\in\Z,l\in\Z+\delta.$$
For any $\g$-module $V$ and any $\mu\in\C^*$, define the $\g$-module $V^\mu$ by
$$x\cdot v=\sigma_\mu(x)v,x\in\g,v\in V.$$

 \begin{definition}\label{def.2.3} An $\S$-module $V$ is called to be of  {\bf level} $\ell$ and with  {\bf central charge} $\dot c$ if $z, c$ acts on $V$ as scalar $\ell, \dot c$ respectively.\end{definition}

            \begin{definition}\label{def.4} An $\S$-module $V$ is called a {\bf weight module} if $L_0$ acts on $V$ diagonally, i.e., $V=\oplus_{\alpha\in\C}V_\alpha$, where $V_{\alpha}=\{v\in V:L_0\cdot v=\alpha v\}$. For a weight $\S$-module $V$, write
                        $$\supp(V)=\{\alpha\in\C:  V_\alpha\neq 0\},$$
                        which is called the {\bf support set} of $V$. A weight $\S$-module $V$ is called a {\bf highest (resp. lowest) weight module} if $\supp(V)\subset\alpha+\Z_+$ (resp. $\alpha-\Z_+$) for some $\alpha\in\C$.\end{definition}

\section{Representations of $F(\delta)$}\label{SecF}
\begin{definition}
            Let $V$ be a module over a Lie superalgebra $\g$, $\h$ be a subalgebra of $\g$, and $S$ be a subset of $V$. We denote
            $$\Ann_V(\h)=\{v\in V: \h v=0\},\,\,\,\Ann_\g(S)=\{x\in\g:xS=0\}.$$
\end{definition}

\begin{definition}
            Let $\g=\oplus_{i\in\Z}\g_i$ be a $\Z$-graded Lie superalgebra. A $\g$-module $V$ is called a {\bf  smooth module} if for any $v\in V$ there exists $n\in\Z_+$ such that $\g_iv=0$ for $i>n$. The category of smooth modules over $\g$ will be denoted by $\mathfrak R_{\g}$.
\end{definition}

Smooth modules    for affine Kac-Moody algebras $\g$   were introduced and studied by D. Kazhdan and G. Lusztig  in  \cite{KL}.

Let $\mathscr V(\delta)$ be the exterior algebra generated by the odd elements $\xi_i,i\in\Z_++\delta$. Then $\mathscr V(\delta)$ has a basis
$$1,\xi_{i_1}\cdots\xi_{i_k},\,\,\,0\leqslant i_1<\dots<i_k,k\in\N.$$
Now $\mathscr V(\delta)$ has an irreducible $F(\delta)$-module structure defined by
$$\aligned &\psi_n\cdot v=\partial/\partial\xi_n\cdot v,\\ &\psi_{-n}\cdot v
 = \xi_n\cdot v,\\ &\psi_0\cdot v  = \frac{1}{\sqrt 2}(\xi_0+\partial/\partial{\xi_0})\cdot v\,\, (\text{if } \delta=0),\\ &z\cdot v= v,\endaligned$$
where $n\in(\Z_++\delta)\setminus\{0\},v\in\mathscr V(\delta)$. Write $\mathscr V:=\mathscr V(0)$ for short.

\begin{definition}
            For $m\in\Z_+$, define in the same manner above the $F_{[m]}$-submodule $\mathscr V_{[m]}$ of $\mathscr V$ by
            $$\mathscr V_{[m]}=\span\{1,\xi_{i_1}\cdots\xi_{i_k}\ :\ 0\leqslant i_1<\dots<i_k\leqslant m\}.$$
            Then $\mathscr V_{[m]}$ is a finite-dimensional simple $F_{[m]}$-module.
\end{definition}

\begin{lemma}\label{lemrF0}
            \begin{itemize}
                        \item[(1)] Suppose $V$ is a simple $F_{[m]}$-module such that $zV=0$. Then $F_{[m]}V=0 $. Consequently $\dim V=1$.
                        \item[(2)] Suppose $V$ is a simple $F(\delta)$-module such that $zV=0$. Then $F(\delta)V=0 $.  Consequently $\dim V=1$.
            \end{itemize}
\end{lemma}
\begin{proof}Since $z=0$, $\psi_i V$ is a submodule of $V$, which implies $\psi_i V=0$ or $\psi_i V=V$. The latter cannot occur since $\psi_i^2 V=0$. \end{proof}

\begin{lemma}\label{lemrF1}
            \begin{itemize}
                        \item[(1)] Suppose $V$ is a simple $F_{[m]}$-module such that $z$ acts on $V$ as a scalar $\lambda\in\C^*$. If there is a nonzero vector  $v\in V_{\bar 0}$ such that $\psi_k\cdot v=0$ for all $k\in\{1,\dots,m\}$, then $V\cong\mathscr V_{[m]}^{\sqrt\lambda}$ for some $\sqrt\lambda\in\C$ with $(\sqrt\lambda)^2=\lambda$.
                        \item[(2)]Suppose $V$ is a simple $F$-module  such that $z$ acts on $V$ as a scalar $\lambda\in\C^*$. If there is a nonzero vector  $v\in V_{\bar 0}$ such that $\psi_k\cdot v=0$ for all $k\in\N$, then $V\cong\mathscr V^{\sqrt\lambda}$ for some $\sqrt\lambda\in\C$ with $(\sqrt\lambda)^2=\lambda$.
            \end{itemize}
\end{lemma}
\begin{proof}
            (1) It suffices to prove that $V^{\frac{1}{\sqrt\lambda}}\cong\mathscr V_{[m]}$. Note that $z$ acts on $V^{\frac{1}{\sqrt\lambda}}$ as the identity map. Without loss of generality, assume that $\lambda=1$.

            Since $\psi_0^2\cdot v=\frac{1}{2}[\psi_0,\psi_0]\cdot v=\frac{1}{2}v$, $\C v\oplus\C\psi_0v$ is a simple $F_{[0]}$-module. Also, the fact that $\psi_k\cdot(\C v\oplus\C\psi_0v)=0$ for all $k=1,\dots,m$ implies that  $\C v\oplus\C\psi_0v$ is a simple $F_{[0]}\oplus F_{[m]+}$-module, where $F_{[m]+}=F_{[m]}\cap F_+$. So $V$ is a simple quotient of $\rm{Ind}_{F_{[0]}\oplus F_{[m]+}}^{F_{[m]}}(\C v\oplus\C\psi_0v)$.

            Let $f:\C v\oplus\C\psi_0v\rightarrow\mathscr V_{[0]}$ be a linear map defined by
            $f(v)=1,f(\psi_0v)=\frac{1}{\sqrt 2}\xi_0$. Then $f$ is an isomorphism of $(F_{[0]}\oplus F_{[m]+})$-modules. Thus, $f$ induces an $F_{[m]}$-module homomorphism
            \begin{equation*}\begin{split}
            \tilde f:\rm{Ind}_{F_{[0]}\oplus F_{[m]+}}^{F_{[m]}}(\C v\oplus\C\psi_0v)&\rightarrow\mathscr V_{[m]}\\
            \psi_{i_1}\dots\psi_{i_k}v&\mapsto\xi_{-i_1}\dots\xi_{-i_k},\\
            \psi_{i_1}\dots\psi_{i_k}\psi_0v&\mapsto\frac{1}{\sqrt 2}\xi_{-i_1}\dots\xi_{-i_k}\xi_0,i_1<\dots<i_k<0.\\
            \end{split}\end{equation*}
            Note that $\tilde f$ maps a basis to a basis. So $\tilde f$ is an $F_{[m]}$-module isomorphism. Since $\mathscr V_{[m]}$ is a simple $F_{[m]}$-module, we have $\rm{Ind}_{F_{[0]}\oplus F_{[m]+}}^{F_{[m]}}(\C v\oplus\C\psi_0v)$ is a simple module and $V\cong\mathscr V_{[m]}$.

            (2) This could be proved similarly.
\end{proof}

\begin{lemma}\label{rFm}
            \begin{itemize}
                        \item[(1)]Suppose $V$ is a simple $F_{[m]}$-module for some $m\in\Z_+$ such that $z$ acts on $V$ as a scalar $\lambda\in\C^*$. Then $V\cong\mathscr V_{[m]}^{\sqrt\lambda}$ for some $\sqrt\lambda\in\C$ with $(\sqrt\lambda)^2=\lambda$.
                        \item[(2)]Suppose $V$ is a simple smooth $F$-module such that $z$ acts on $V$ as a scalar $\lambda\in\C^*$. Then $V\cong\mathscr V^{\sqrt\lambda}$ for some $\sqrt\lambda\in\C$ with $(\sqrt\lambda)^2=\lambda$.
            \end{itemize}
\end{lemma}
\begin{proof}
            (1) Note that $\psi_i^2\cdot V=0$ and that the action of $\psi_i$ and $\psi_j$ are anti-commutative for all $i,j\in\{1,\dots,m\}$. So there exists a nonzero homogeneous vector $v\in V$ such that $\psi_k\cdot v=0,k=1,\dots,m$. Replace $v$ with $\psi_0v$ if necessary, and assume that $v\in V_{\bar 0}$. Then (1) follows from Lemma \ref{lemrF1} (1).

            (2) Let $v$ be a nonzero homogeneous vector in $V$. Since $V$ is a smooth $F$-module, there exists an $m\in\N$ such that $\psi_k\cdot v=0$ for any $k>m$. Then $\psi_k,k>m,$ acts trivially on the finite-dimensional $F_{[m]}$-module $U(F_{[m]})v$. Clearly,  $U(F_{[m]})v$ has a nonzero even vector $w$ such that $\psi_k\cdot w=0,k=1,\dots,m$. So $\psi_k\cdot w=0$ for all $k\in\N$. Then (2) follows from Lemma \ref{lemrF1} (2).
\end{proof}

With similar arguments, we have the following result.
\begin{lemma}
Suppose $V$ is a simple smooth $F(\frac{1}{2})$-module such that $z$ acts on $V$ as a scalar $\lambda\in\C^*$. Then $V\cong\mathscr V(\frac{1}{2})^{\sqrt\lambda}$ for some $\sqrt\lambda\in\C$ with $(\sqrt\lambda)^2=\lambda$.
\end{lemma}

Note that the Fermion algebra $F_{[m]}$ has very good properties such that every $F_{[m]}$-module with nonzero action of the center over $F_{[m]}$ is completely reducible.

\begin{lemma}\label{cr1}
            Let $m\in\Z_+$, and $V$ be an $F_{[m]}$-module such that $z$ acts on $V$ as a scalar $\lambda\in\C^*$. Then $V$ is completely reducible.
\end{lemma}
\begin{proof}
            $\mathbf{Claim}$. $U(F_{[m]})v$ is completely reducible for any homogeneous $v\in V$.

            For any $k\in\{1,\dots,m\}$, let $X_k=\{\psi_k,\psi_{-k}\psi_k\},Y_k=\{\psi_{-k},\psi_k\psi_{-k}\}$. Then for any $Z_k\in\{X_k,Y_k\},k=1,\dots,m$, the subspace
            \begin{equation}\label{eqFmfactor}
            \text{span}\{x_1\dots x_mv,\psi_0x_1\dots x_mv:x_j\in Z_j,j=1,\dots,m\}
            \end{equation}
            is an $F_{[m]}$-submodule of $U(F_{[m]})v$. By Lemma \ref{rFm}(1), any simple $F_{[m]}$-module is isomorphic to $\mathscr V_{[m]}^{\sqrt\lambda}$ for some $\sqrt\lambda\in\C$ with $(\sqrt\lambda)^2=\lambda$. Comparing the dimensions of $\mathscr V_{[m]}^{\sqrt\lambda}$ and (\ref{eqFmfactor}), we have (\ref{eqFmfactor}) is zero or is a simple module. Note that $\psi_{-k}\psi_k+\psi_k\psi_{-k}=z$ in the enveloping algebra for any $k\in\{1,\dots,m\}$. Then
            $$\bigoplus_{Z_1\in\{X_1,Y_1\}}\dots\bigoplus_{Z_m\in\{X_m,Y_m\}}\span\{x_1\dots x_mv,\psi_0x_1\dots x_mv:x_j\in Z_j,j=1,\dots,m\}$$
            contains $v$, and therefore equals $U(F_{[m]})v$. Thus, $U(F_{[m]})v$ is completely reducible.

By the Claim, for any homogeneous $v\in V\setminus\{0\}$, $U(F_{[m]})v$ is a sum of simple modules. Then $V$ is also a sum of simple modules, which means that $V$ is completely reducible.
\end{proof}

\begin{lemma}\label{cr2}
Let $V$ be a smooth $F$-module such that $z$ acts on $V$ as a scalar $\lambda\in\C^*$. Then $V$ is completely reducible.
\end{lemma}
\begin{proof}
It suffices to prove that $U(F)v$ is a sum of simple $F$-modules for any homogeneous  $v\in V$. Let $v$ be a nonzero homogeneous  vector in $V$. Then there exists $m\in\N$ such that $\psi_{m+k}v=0$ for all $k\in\Z_+$. By Lemma \ref{cr1}, $U(F_{[m]})v$ is completely reducible. From Lemma \ref{rFm} (1), there are $x_1,\dots,x_r\in U(F_{[m]})v$ satisfying  $\psi_kx_i=0$ for $k\in\Z_+$ such that $U(F_{[m]})v=U(F_{[m]})x_1+\dots+U(F_{[m]})x_r$. We see that
$$U(F)v=U(F)x_1+\dots+U(F)x_r.$$
Note that each $U(F)x_i$ is isomorphic to $\mathscr V^{\sqrt\lambda}$ for some $\sqrt\lambda\in\C$ with $(\sqrt\lambda)^2=\lambda$. So $U(F)v$ is a sum of simple $F$-modules.
\end{proof}

For any $I\subset\Z_+$, let
$$S_I=\{J\subset\Z_+:|I\setminus J|<\infty,|J\setminus I|<\infty\}.$$
Define the $F$-module $\mathscr V_I$ as follows: $\mathscr V_I$ has a basis $\{\xi_J:  J\in S_I\}$, the parity of $\xi_J$ is defined as $|\xi_J|=\overline{|I\setminus J|+|J\setminus I|}$, and the action of $F$ on $\mathscr V_I$ is given by
\begin{equation*}\begin{aligned}
\psi_k\cdot\xi_J=&\left\{\begin{aligned}(-1)^{|\{i\in J:  i<k\}|}\xi_{J\setminus\{k\}},&\ \text{if}\ k\in J,\\0,&\ \text{if}\ k\notin J,\end{aligned}\right.\\
\psi_{-k}\cdot\xi_J=&\left\{\begin{aligned}(-1)^{|\{i\in J:  i<k\}|}\xi_{J\cup\{k\}},&\ \text{if}\ k\notin J,\\0,&\ \text{if}\ k\in J,\end{aligned}\right.\\
\psi_0\cdot\xi_J=&\left\{\begin{aligned}\frac{1}{\sqrt 2}\xi_{J\setminus\{0\}},&\ \text{if}\ 0\in J,\\\frac{1}{\sqrt 2}\xi_{J\cup\{0\}},&\ \text{if}\ 0\notin J,\end{aligned}\right.\\
z\cdot\xi_J=&\xi_J,
\end{aligned}\end{equation*}
$k\in\N$. Then $\mathscr V_I$ is a simple $F$-module (see the proof of Lemma \ref {L3.8}). Moreover, $\mathscr V_I$ is isomorphic to the $F$-module $\mathscr V$ if $I$ is a finite set, and $\mathscr V_I$ is not a smooth $F$-module if $I$ is not a finite set.

\begin{lemma}\label{L3.8}
	Suppose $V$ is a simple $F$-module such that $z$ acts on $V$ as the identity map. If there is $v\in V_{\bar 0}\setminus\{0\}$ and $I\subset\Z_+$ such that
	$$\psi_l\cdot v=0,\psi_{-k}\cdot v=0,\forall l\in\N\setminus I,k\in I\setminus\{0\},$$
	then $V\cong\mathscr V_I$.
\end{lemma}
\begin{proof}
	Let $f_I$ be the Lie superalgebra isomorphism of $F$ defined by
	$$f_I(\psi_k)=\psi_{-k},f_I(\psi_{-k})=\psi_k,f_I(\psi_l)=\psi_l,\forall k\in I,l\in\Z\setminus(I\cup -I).$$
	Then $f_I^2$ is the identity map on $F$. Let $V^{f_I}$ be the simple $F$-module given by $V^{f_I}=V$ and
	$$x\cdot w=f_I(x)w,\forall x\in F,w\in V.$$
	Then for $v\in V^{f_I}$ we have
	$$\psi_k\cdot v=0,\forall k\in\N.$$
	By Lemma \ref{lemrF1}(2), $V^{f_I}\cong\mathscr V$. So $V\cong\mathscr V^{f_I}$.
	
	Note that $\mathscr V_I$ is a simple $F$-module with the given property. Then $V\cong\mathscr V^{f_I}\cong\mathscr V_I$.
\end{proof}

Suppose $V$ is an $F$-module.
For any $\lambda=(\lambda_0,\lambda_1,\dots)\in\C^{\infty}$, let
$$V_{(\lambda)}=\{v\in V:\psi_{-k}\psi_k\cdot v=\lambda_kv,k\in\Z_+\}.$$
$V$ is called a {\bf weight $F$-module} if $V=\oplus_{\lambda\in\Z^\infty}V_{(\lambda)}$. By definition, for any $I\subset\Z_+$, $\mathscr V_I$ is a weight $F$-module.

\begin{lemma}\label{weightF}
Suppose $V$ is a weight $F$-module such that  $z$ acts as the identity map on $V$.
\begin{itemize}
\item[(1)] $\lambda_k=0$ or $1$ for any $\lambda\in\C^\infty$ with $V_{(\lambda)}\neq 0$ and any $k\in\N$.
\item[(2)] For any $v\in V_{(\lambda)}\setminus\{0\}$, $U(F)v$ is isomorphic to $\mathscr V_I$ for some $I\subset\Z_+$ and $v$ is corresponding to $\xi_I$.
\item[(3)] If $V$ is further simple, then $V$ is isomorphic to $\mathscr V_I$ for some $I\subset\Z_+$.
\end{itemize}
\end{lemma}
\begin{proof} (1)
Suppose $v\in V_{(\lambda)}\setminus\{0\}$ for some $\lambda\in\C^\infty$. For any $k\in\N$,
$$\lambda_k^2v=\psi_{-k}\psi_k\psi_{-k}\psi_kv=\psi_{-k}[\psi_k,\psi_{-k}]\psi_kv=\psi_{-k}\psi_kv=\lambda_kv.$$
So $\lambda_k=0$ or $1$ and (1) is proved.

(2) Replacing $v$ with $\psi_0v$ if necessary, assume that $|v|=\bar 0$.
If $\psi_{-k}\psi_kv=0$, then $\psi_kv=\psi_k\psi_{-k}\psi_kv=0$. If $\psi_{-k}\psi_kv=v$, then $\psi_{-k}v=0$. Let $I=\{k\in\N:\lambda_k=1\}$ and let $J(I)$ be the left ideal of $U(F)$ generated by $\psi_{-i},\psi_j,i\in I,j\in\N\setminus I$. It is easy to see that $U(F)/(J(I)+\langle z-1\rangle)\cong\mathscr V_I$. Then $U(F)v\cong\mathscr V_I$ and $v$ is corresponding to $\xi_I$. So (2) is proved.

Part (3) follows from (2).
\end{proof}

\section{smooth modules over $\S(\delta)$}


The $F(\delta)$-module $\mathscr V(\delta)$ constructed last section has an $\S(\delta)$-module structure as follows.
For $k\in\Z$, we have the following vertex operator
\begin{equation}\label{barL}
\bar L_k=\frac{(1 - 2\delta)}{16}\delta_{k,0}+\frac{1}{2}\sum_{j\in\Z+\delta}j:\psi_{-j}\psi_{j+k}:
\end{equation}
as an operation on $\mathscr V(\delta)$, where the normal ordering is defined by
\begin{equation*}\begin{aligned}
:\psi_j\psi_k:=&\psi_j\psi_k&\ \text{if}\ k\geqslant j\\
=&-\psi_k\psi_j&\ \text{if}\ k<j.
\end{aligned}\end{equation*}
\begin{lemma}\label{KRR}\cite[Proposition 3.7]{KRR}
Define the action of $\Vir$ on $\mathscr V(\delta)$ by
$$L_k\cdot v=\bar L_k\cdot v,\,\,\,c\cdot v=\frac{1}{2}v,\,\,\,v\in\mathscr V(\delta).$$
Then $\mathscr V(\delta)$ is an $\S(\delta)$-module, i.e.,
\begin{equation}\label{eqbrac}\begin{aligned}
&[\bar L_m,\bar L_n]=(m-n)\bar L_{m+n}+\delta_{m,-n}\frac{m^3-m}{24},\\
&[\bar L_m,\psi_n]=(-n-\frac{m}{2})\psi_{m+n}.
\end{aligned}\end{equation}
\end{lemma}
Note that the action of $L_0$ on $\mathscr V(\delta)$ is $\frac{(1 - 2\delta)}{16}+\sum_{j\in\Z_++\delta\setminus\{0\}}j\psi_{-j}\psi_j$. Then $\xi_{i_1}\cdots\xi_{i_k}$ is an eigenvector of $L_0$ of eigenvalue $\frac{(1 - 2\delta)}{16}+i_1+\dots+i_k$. Also, for any $v\in\mathscr V(\delta)$, $L_k\cdot 1=0=\psi_k\cdot 1$ for sufficiently large $k$.



\begin{lemma}\label{LZlem8}\cite[Lemma 8]{LZ}
Let $W$ be a module over a Lie superalgebra $L$, $\g$ be a subalgebra of $L$, and $B$ be a $\g$-module. Then the $L$-module homomorphism $\Ind_\g^L(W\otimes B)\rightarrow W\otimes\Ind_\g^L(B)$ induced from the inclusion map $W\otimes B\rightarrow W\otimes\Ind_\g^L(B)$ is an $L$-module isomorphism.
\end{lemma}

For any $\Vir$-module $V$, define the $\S(\delta)$-module structure on $V$ by letting $F(\delta)\cdot V=0$. Denote the resulting module by $V^{\S(\delta)}$.

\begin{theorem}\label{smooth1}
Let $M$ be a simple smooth $\S$-module of level $\lambda\in\C^*$. Then $M\cong\mathscr V^{\sqrt\lambda}\otimes V^\S$ for some simple smooth $\Vir$-module $V$ and some $\sqrt\lambda\in\C$ with $(\sqrt\lambda)^2=\lambda$, where  $\mathscr V^{\sqrt\lambda}$ is considered as an $\S$-module.
\end{theorem}
\begin{proof}
Let $v$ be a nonzero homogeneous vector in $M$. Then there is an $m\in\N$ such that $\psi_j\cdot v=0$ for all $j>m$. By Lemma \ref{rFm}, $U(F_{[m]})v$ has a simple $F_{[m]}$-submodule that is isomorphic to $\mathscr V_{[m]}^{\sqrt\lambda}$. Then $U(F_{[m]})v$ has a nonzero even vector $w$ with $\psi_j\cdot w=0,j\in\N$. So $U(F)w$ is a simple $F$-quotient of $\mathrm{Ind}_{F^{(0)}}^F(\C w\oplus\C\psi_0 w)\cong\mathscr V^{\sqrt\lambda}$. Thus, $U(F)w\cong\mathscr V^{\sqrt\lambda}$ as $F$-modules.

Let $n\in\N$ with $n\geqslant 2m$ such that $L_i\cdot w=0$ for all $i\geqslant n$. Then $U(F)w$ is a simple module over $\S^{(n,-\infty)}$. Suppose $c$ acts on $M$ as a scalar $\mu$. Define the one-dimensional $\S^{(n,-\infty)}$-module $\C u$ by
$$L_i\cdot u=F\cdot u=0, \,\,\,c\cdot u=(\mu-\frac{1}{2})u,i\geqslant n.$$
Then $U(F)w\cong\mathscr V^{\sqrt\lambda}\otimes\C u$ as $\S^{(n,-\infty)}$-modules. By Lemma \ref{LZlem8},
$$\Ind_{\S^{(n,-\infty)}}^\S U(F)w\cong\Ind_{\S^{(n,-\infty)}}^\S(\mathscr V^{\sqrt\lambda}\otimes\C u)\cong\mathscr V^{\sqrt\lambda}\otimes\Ind_{\S^{(n,-\infty)}}^\S\C u.$$
So $M$ is a simple quotient of $\mathscr V^{\sqrt\lambda}\otimes\Ind_{\S^{(n,-\infty)}}^\S\C u$. Note that $\Ind_{\S^{(n,-\infty)}}^\S\C u$ is a smooth $\Vir$-module. It suffices to prove that any submodule of $\mathscr V^{\sqrt\lambda}\otimes\Ind_{\S^{(n,-\infty)}}^\S\C u$ is of the form $\mathscr V^{\sqrt\lambda}\otimes W$, where $W$ is a $\Vir$-submodule of $\Ind_{\S^{(n,-\infty)}}^\S\C u$.

Let $M'$ be a nonzero submodule of $\mathscr V^{\sqrt\lambda}\otimes\Ind_{\S^{(n,-\infty)}}^\S\C u$. We claim that $M'=\mathscr V^{\sqrt\lambda}\otimes V$, where $V=\{x\in\Ind_{\S^{(n,-\infty)}}^\S\C u:  \mathscr V^{\sqrt\lambda}\otimes\C x\subset M'\}$. It follows that $V$ is a $\Vir$-submodule of $\Ind_{\S^{(n,-\infty)}}^\S\C u$ and we are done.

Let $a$ be any nonzero homogeneous vector in $M'$. Then $a$ is of the form $\sum_{i=1}^k v_i\otimes u_i$, where $v_i\in\mathscr V^{\sqrt\lambda},u_i\in(\Ind_{\S^{(n,-\infty)}}^\S\C u)\setminus\{0\}$ and $v_1,\dots,v_k$ are linearly independent. Since $u_1,\dots,u_k$ are of the same parity, so $v_1,\dots,v_k$ are of the same parity. By repeatedly letting $\psi_i,i\in\Z_+$, act on $a$, we get $1\otimes u_i\in M'$ up to a scalar for some $i\in\{1,\dots,k\}$. So $\mathscr V^{\sqrt\lambda}\otimes u_i\subset M'$ and $u_i\in V$. Using induction on $k$, we have $u_1,\dots,u_k\in V$. Therefore, $M'=\mathscr V^{\sqrt\lambda}\otimes V$.
\end{proof}

With similar arguments, we have the following result.

\begin{theorem}\label{smooth2}
Let $M$ be a simple smooth $\S(\frac{1}{2})$-module  of level $\lambda\in\C^*$. Then $M\cong\mathscr V(\frac{1}{2})^{\sqrt\lambda}\otimes V^{\S(\frac{1}{2})}$ for some simple smooth $\Vir$-module $V$ and some $\sqrt\lambda\in\C$ with $(\sqrt\lambda)^2=\lambda$, where  $\mathscr V(\frac{1}{2})^{\sqrt\lambda}$ is considered as an $\S(\frac{1}{2})$-module.
\end{theorem}

\begin{definition}
	A module $M$ over a  Lie superalgebra $A$ is called strictly simple if it is a simple module over the   algebra $A$ (forgetting the $\mathbb Z_2$-gradation).
\end{definition}

Let ${\mathfrak g}={\mathfrak a}\ltimes{\mathfrak b}$ be a Lie superalgebra where ${\mathfrak a}$ is a subalgebra  of ${\mathfrak g}$ and ${\mathfrak b}$ is an ideal of ${\mathfrak g}$. Let $M$ be a  ${\mathfrak g}$-module with a   ${\mathfrak b}$-submodule $H$ so that the ${\mathfrak b}$-submodule structure on $H$ can be extended to a ${\mathfrak g}$-module structure on $H$. We denote this ${\mathfrak g}$-module by $H^{\mathfrak g}$. For any ${\mathfrak a}$-module   $U$, we can make it into a ${\mathfrak g}$-module by ${\mathfrak b}U=0$. We denote this ${\mathfrak g}$-module by $U^{\mathfrak g}$. Then by Lemma 3.3 in \cite{LPXZ} we have

\begin{lemma}\label{simple-app} Let ${\mathfrak g}={\mathfrak a}\ltimes{\mathfrak b}$ be a countable dimensional Lie superalgebra.
	Let $M$ be a  simple ${\mathfrak g}$-module
	with a strictly simple  ${\mathfrak b}$-submodule $H$ so that an $H^{\mathfrak g}$ exists.
	Then $M\cong U^{\mathfrak{g}}\otimes H^{\mathfrak{g}}$ as $\mathfrak{g}$-modules for some  simple  ${\mathfrak a}$-module $U$.
\end{lemma}
Note that Theorem \ref{smooth1} can not be proved directly by using Lemma \ref{simple-app} because the $F$-module $\mathscr V^{\sqrt\lambda}$ is not strictly simple.

We know that  simple smooth $\Vir$-modules were classified in \cite[Theorem 2]{MZ}.

Applying Theorem \ref{smooth1} to highest weight $\S$-modules we obtain the following theorem.


%
%
%
%

\begin{theorem}\label{hwm2} Let $V$ be a simple highest weight $\S$-module. \begin{itemize}
	\item[(1)]  If $z$ acts on $V$ as a scalar $\lambda\in\C^*$, then $V\cong\mathscr V^{\sqrt\lambda}\otimes M^\S$ for some simple highest weight $\Vir$-module $M$ and some $\sqrt\lambda\in\C$ with $(\sqrt\lambda)^2=\lambda$, where  $\mathscr V^{\sqrt\lambda}$ is considered as an $\S$-module.
\item[(2)]	
If $zV=0$, then $FV=0$ and $V$ is a simple highest weight $\Vir$-module.
\end{itemize}
\end{theorem}
\begin{proof} (1) This follows from Theorem \ref{smooth1} since highest weight $\Vir$-module are smooth modules.

(2) Let $v$ be a nonzero homogeneous highest weight vector of $V$. Then $\psi_k\cdot v=0$ for all $k>0$. Replacing $v$ with $\psi_0v$ if necessary, assume that $\psi_0v=0$. By the simplicity of $V$ and the PBW theorem, $V=U(\Vir_-)U(F_-)v$. Then
$$F\cdot V=F\cdot U(\Vir_-)U(F_-)v\subset U(\Vir_-)U(F_-)F_-v.$$
Considering the support set of $U(\Vir_-)U(F_-)F_-v$ we have $v\notin F\cdot V$. So $F\cdot V=0$ and we are done.
\end{proof}

We remark that in this paper we are not able to determine all simple smooth $\S(\delta)$-modules $V$  of level $0$.
%
%
%

\section{Harish-Chandra modules}
%

A weight $\S(\delta)$-module is called a Harish-Chandra module if its weight spaces are all finite-dimensional. A weight $\S(\delta)$-module is called a bounded module if the dimensions of its weight spaces are uniformly bounded.
\begin{lemma}\label{HC}
A simple Harish-Chandra $\S$-module is either a highest weight module, a lowest weight module or a bounded module.
\end{lemma}
\begin{proof}
Suppose $M$ is a simple Harish-Chandra $\S$-module that is not bounded. Then $V=U(\Vir)\cdot(\oplus_{\alpha\neq 0}M_\alpha)$ is the unique minimal $\Vir$-submodule of $M$ such that $M/V$ is a trivial $\Vir$-module. Let $T$ be the maximal trivial $\Vir$-submodule of $V$. Then $\bar V=V/T$ is an Harish-Chandra $\Vir$-module such that $\bar V$ and $\bar V^*$ do not contain trivial $\Vir$-submodules, where $\bar V^*$ is the restricted dual of $\bar V$. By Proposition 3.3 in \cite{MP}, $\bar V=V^-\oplus B\oplus V^+$, where $V^+$ (resp. $V^-$) is the maximal $\Vir$-submodule of $\bar V$ with lower (resp. upper) bounded weights, and $B$ is a bounded $\Vir$-submodule of $\bar V$.

Without loss of generality, assume $V^+\neq 0$. Then $V$ contains a nonzero homogeneous weight vector $v$ such that $L_k\cdot v=0$ for all $k\in\N$, and $U(\Vir)v$ is not a trivial module. Then $U(\Vir)v$ has a nontrivial simple $\Vir$-subquotient that is a highest weight $\Vir$-module. So $U(\Vir)v$ is not bounded. Suppose $v\in V_h$. Then there is a $k\in\N$ such that $$\dim(U(\Vir)v)_{h+k}>2\dim M_h.$$
Then there exists $w\in(U(\Vir)v)_{h+k}$ such that $L_k\cdot w=\psi_k\cdot w=0$. Also, for any $j\in\N$, $L_{k+j}\cdot w=0$. So
$$0=[L_{k+j},\psi_k]\cdot w=-\frac{1}{2}(3k+j)\psi_{2k+j}\cdot w.$$
Thus, $\S_n\cdot w=0$ for all $n\geqslant 2k$. By Lemma 1.6 in \cite{M}, $M$ has a highest weight. Then $M$ is a highest weight module.
\end{proof}
\begin{lemma}\label{bounded1}
Suppose $V$ is a simple bounded $\S$-module. Then $zV=0$.
\end{lemma}
\begin{proof}
Since $V$ is simple, $z$ acts on $V$ as a scalar. Suppose that $zV\neq 0$ and, without loss of generality, assume that $z$ acts on $V$ as the identity map.

Let $\lambda\in\supp(V)$. Then $\dim V_\lambda<\infty$ and $\psi_{-k}\psi_{k}\cdot V_\lambda\subset V_\lambda$ for all $k\in\Z_+$. So $V_\lambda$ contains a nonzero homogeneous common eigenvector $v$ of $\psi_{-k}\psi_k,k\in\Z_+$. By Lemma \ref{weightF}, $U(F)v$ is isomorphic to $\mathscr V_I$ for some $I\subset\Z_+$ and $v$ is corresponding to $\xi_I$. Identify elements in $\mathscr V_I$ with that in $U(F)v$. For any $J\in S_I$, $\xi_J\in V_\mu$, where
$$\mu=\lambda+\sum_{k\in J\setminus I}k-\sum_{l\in I\setminus J}l.$$
Then $U(F)v$ is a weight $\C L_0$-module.

We will now prove that $U(F)v$ is not bounded, which will lead to a contradiction and the lemma follows. Suppose $U(F)v$ is bounded with maximal weight space dimension $d$. Let $\mu\in\supp(V)$ with $\dim(U(F)v)_\mu=d$. Then $(U(F)v)_\mu$ has a basis $\xi_{J_1},\dots,\xi_{J_d}$ satisfying
$$\mu=\lambda+\sum_{k\in J_i\setminus I}k-\sum_{l\in I\setminus J_i}l.$$
Let $m\in\N$ such that $J_i\setminus\{0,\dots,m\}$ are the same for all $i\in\{1,\dots,d\}$.

If $I\subset\{0,\dots,m\}$, then $\xi_{J_i\cup\{m+1,m+2\}},\xi_{J_i\cup\{2m+3\}}\in(U(F)v)_{\mu+2m+3}$ for all $i\in\{1,\dots,d\}$. So
$$\dim(U(F)v)_{\mu+2m+3}\geqslant 2d,$$
which is a contradiction. Assume that $I\not\subset\{0,\dots,m\}$ and, without loss of generality, that $m+1\in I$.

\textbf{Case 1}. There is $n\in\N\setminus\{0,\dots,m+1\}$ such that $n\in I$ and $m+1+n\in I$.

Then $\xi_{J_i\setminus\{m+1,n\}},\xi_{J_i\setminus\{m+1+n\}}\in(U(F)v)_{\mu-m-1-n}$ for all $i\in\{1,\dots,d\}$. So
$$\dim(U(F)v)_{\mu-m-1-n}\geqslant 2d,$$
which is a contradiction.

\textbf{Case 2}. There is $n\in\N\setminus\{0,\dots,m+1\}$ such that $n\notin I$ and $m+1+n\in I$.

Then $\xi_{J_i\setminus\{m+1\}},\xi_{(J_i\setminus\{m+1+n\})\cup\{n\}}\in(U(F)v)_{\mu-m-1}$ for all $i\in\{1,\dots,d\}$. So
$$\dim(U(F)v)_{\mu-m-1}\geqslant 2d,$$
which is a contradiction.

\textbf{Case 3}. For any $n\in\N\setminus\{0,\dots,m+1\}$ we have $m+1+n\notin I$.

Then $(\N\setminus\{0,\dots,2m+2\})\cap I=\vn$, i.e., $I\subset\{0,\dots,2m+2\}$. As the proof above, this also leads to a contradiction.

Thus, $U(F)v$ is not bounded.
\end{proof}

Now we recall the following results from \cite{CLW, DCL}.

\begin{lemma}\label{bounded2}
Suppose $V$ is a simple bounded $\S$-module  of level $0$. Then $FV=0$ and $V$ is a simple $\Vir$-module of the intermediate series.
\end{lemma}

We are ready to state the classification for simple Harish-Chandra $\S$-modules.
\begin{theorem}\label{HC1}
Any simple Harish-Chandra $\S$-module is isomorphic to one of the following $\S$-modules.
\begin{itemize}\item[(1)]
$V^\S$, where $V$ is a simple Harish-Chandra $\Vir$-module and $ FV=0$;
\item[(2)]   $\mathscr V^\lambda\otimes V^\S$, where $\lambda\in\C^*$ and $V$ is a simple highest weight $\Vir$-module, and  $\mathscr V^{ \lambda}$ is considered as an $\S$-module;
\item[(3)]  the  restricted dual of the module in \rm{(2)}.
\end{itemize}
\end{theorem}
\begin{proof}
This follows from Lemma \ref{HC}, Lemma \ref{bounded1}, Lemma \ref{bounded2}   and Theorem \ref{hwm2}.
\end{proof}

 We have similar results for the Fermion-Virasoro  algebra $\S(\frac{1}{2})$.

\begin{theorem}\label{HC2}
Any simple Harish-Chandra $\S(\frac{1}{2})$-module is isomorphic to one of the following $\S(\frac{1}{2})$-modules.
\begin{itemize}\item[(1)]
$V^{\S(\frac{1}{2})}$, where $V$ is a simple Harish-Chandra $\Vir$-module and $ F(\frac{1}{2})V=0$;
\item[(2)]   $\mathscr V(\frac{1}{2})^\lambda\otimes V^{\S(\frac{1}{2})}$, where $\lambda\in\C^*$ and $V$ is a simple highest weight $\Vir$-module, and  $\mathscr V(\frac{1}{2})^{ \lambda}$ is considered as an $\S(\frac{1}{2})$-module;
\item[(3)]  the  restricted dual of the module in \rm{(2)}.
\end{itemize}
\end{theorem}

Note that the modules in (2) and (3) of Theorems \ref{HC1}, \ref{HC2} do not occur in the classifications for the algebras in \cite{CLW,DCL}.

\section{free $\C[L_0]$-modules over Fermion-Virasoro algebras}

Our purpose in this section is to construct modules over $\S(0)$ and $\S(\frac12)$ that are free $\C[L_0]$-modules by restriction.

For $\lambda\in\C^*,a\in\C$, define the free $\C[L_0]$-module $\Omega(\lambda,a):=\C[x]$ of rank $1$ over $\Vir$ by
$$c\cdot f(x)=0,\ L_m\cdot f(x)=\lambda^m(x+ma)f(x+m),\ m\in\Z,f(x)\in\C[x].$$
We have the following lemma.

\begin{lemma}\cite[Theorem 3]{TZ}\label{Ola}
Any $\Vir$-module that is a free $\C[L_0]$-module of rank $1$   is isomorphic to $\Omega(\lambda,a)$ for some $\lambda\in\C^*$ and $a\in\C$.
\end{lemma}

For $\lambda\in\C^*$ and $a,b\in\C$, let $\Omega_{\S(\delta)}(\lambda,a,b):=\C[x_0]\oplus\C[x_1]$ with $\Omega_{\S(\delta)}(\lambda,a,b)_{\bar 0}=\C[x_0]$ and $\Omega_{\S(\delta)}(\lambda,a,b)_{\bar 1}=\C[x_1]$. Fix some $\sqrt\lambda\in\C^*$ with $(\sqrt\lambda)^2=\lambda$, and set $\lambda^r=\lambda^{r-\frac{1}{2}}\sqrt\lambda$ for any $r\in\frac12+\Z$.
Define an $\S( \delta)$-action on $\Omega_{\S(\delta)}(\lambda,a,b)$ as follows:
\begin{equation*}\begin{split}
		&L_m\cdot f(x_0)=\lambda^m(x_0+ma)f(x_0+m),\\
		&L_m\cdot g(x_1)=\lambda^m(x_1-\frac{m}{2}+ma)g(x_1+m),\\
		&\psi_r\cdot f(x_0)=b\lambda^rf(x_1+r),\\
		&\psi_r\cdot g(x_1)=0,\\
		&c\cdot\Omega_{\S(\delta)}(\lambda,a,b)=z\cdot\Omega_{\S(\delta)}(\lambda,a,b)=0,
\end{split}\end{equation*}
for any $m\in\Z, r\in\delta+\Z$ and any $f(x_0)\in\C[x_0],g(x_1)\in\C[x_1]$.

For $\lambda\in\C^*$ and $a,b\in\C$, let $\Omega'_{\S(\delta)}(\lambda,a,b):=\C[y_0]\oplus\C[y_1]$ with $\Omega'_{\S(\delta)}(\lambda,a,b)_{\bar 0}=\C[y_0]$ and $\Omega'_{\S(\delta)}(\lambda,a,b)_{\bar 1}=\C[y_1]$. Fix some $\sqrt\lambda\in\C^*$ with $(\sqrt\lambda)^2=\lambda$, and set $\lambda^r=\lambda^{r-\frac{1}{2}}\sqrt\lambda$ for any $r\in\frac 12+\Z$.
Define an $\S(\delta)$-action on $\Omega'_{\S(\delta)}(\lambda,a,b)$ as follows:
\begin{equation*}\begin{split}
		&L_m\cdot f(y_0)=\lambda^m(y_0+ma)f(y_0+m),\\
		&L_m\cdot g(y_1)=\lambda^m(y_1-\frac{3m}{2}+ma)g(y_1+m),\\
	\end{split}\end{equation*}
\begin{equation*}\begin{split}
		&\psi_r\cdot f(y_0)=b\lambda^r(y_1-2ra+3r)f(y_1+r),\\
		&\psi_r\cdot g(y_1)=0,\\
		&c\cdot\Omega'_{\S(\delta)}(\lambda,a,b)=z\cdot\Omega'_{\S(\delta)}(\lambda,a,b)=0
\end{split}\end{equation*}
for any $m\in\Z, r\in\delta+\Z$ and any $f(y_0)\in\C[y_0],g(y_1)\in\C[y_1]$.

For $\lambda\in\C^*$ and $b\in\C$, let $\Omega''_{\S(\delta)}(\lambda,b):=\C[z_0]\oplus\C[z_1]$ with $\Omega''_{\S(\delta)}(\lambda,b)_{\bar 0}=\C[z_0]$ and $\Omega''_{\S(\delta)}(\lambda,b)_{\bar 1}=\C[z_1]$. Fix some $\sqrt\lambda\in\C^*$ with $(\sqrt\lambda)^2=\lambda$, and set $\lambda^r=\lambda^{r-\frac{1}{2}}\sqrt\lambda$ for any $r\in\frac 12+\Z$. Define an $\S(\delta)$-action on $\Omega''_{\S(\delta)}(\lambda,b)$ as follows:
\begin{equation*}\begin{split}
&L_m\cdot f(z_0)=\lambda^m(z_0+m)f(z_0+m),\\
&L_m\cdot g(z_1)=\lambda^m(z_1-\frac 32m)g(z_1+m),\\
&\psi_r\cdot f(z_0)=b\lambda^r(z_1+3r)(z_1+r)f(z_1+r),\\
&\psi_r\cdot g(z_1)=0,\\
&c\cdot \Omega''_{\S(\delta)}(\lambda,b)=z\cdot \Omega''_{\S(\delta)}(\lambda,b)=0
\end{split}\end{equation*}
for any $m\in\Z, r\in\delta+\Z$ and any $f(z_0)\in\C[z_0],g(z_1)\in\C[z_1]$.

For $\lambda\in\C^*$ and $b\in\C$, let $\tilde\Omega''_{\S(\delta)}(\lambda,b):=\C[\tilde z_0]\oplus\C[\tilde z_1]$ with $\tilde\Omega''_{\S(\delta)}(\lambda,b)_{\bar 0}=\C[\tilde z_0]$ and $\tilde\Omega''_{\S(\delta)}(\lambda,b)_{\bar 1}=\C[\tilde z_1]$.Fix some $\sqrt\lambda\in\C^*$ with $(\sqrt\lambda)^2=\lambda$, and set $\lambda^r=\lambda^{r-\frac{1}{2}}\sqrt\lambda$ for any $r\in\frac 12+\Z$. Define an $\S(\delta)$-action on $\tilde\Omega''_{\S(\delta)}(\lambda,b)$ as follows:
\begin{equation*}\begin{split}
&L_m\cdot f(\tilde z_0)=\lambda^m(\tilde z_0+\frac 52m)f(\tilde z_0+m),\\
&L_m\cdot g(\tilde z_1)=\lambda^m\tilde z_1g(\tilde z_1+m),\\
&\psi_r\cdot f(\tilde z_0)=b\lambda^r\tilde z_1(\tilde z_1-2r)f(\tilde z_1+r),\\
&\psi_r\cdot g(\tilde z_1)=0,\\
&c\cdot \tilde\Omega''_{\S(\delta)}(\lambda,b)=z\cdot \tilde\Omega''_{\S(\delta)}(\lambda,b)=0
\end{split}\end{equation*}
for any $m\in\Z, r\in\delta+\Z$ and any $f(\tilde z_0)\in\C[\tilde z_0],g(\tilde z_1)\in\C[\tilde z_1]$.

\begin{lemma}
	$\Omega_{\S(\delta)}(\lambda,a,b),\Omega'_{\S(\delta)}(\lambda,a,b),\Omega''_{\S(\delta)}(\lambda,b)$ and $\tilde\Omega''_{\S(\delta)}(\lambda,b)$ are $\S(\delta )$-modules.
\end{lemma}

	\begin{proof}
Clearly, both $\C[x_0]$ and $\C[y_0]$ are isomorphic to the $\Vir$-module $\Omega(\lambda,a)$ defined before Lemma \ref{Ola}. 
For any $m,n\in\delta+\Z,x\in\S(\delta)$ and any $f(x_0)\in\C[x_0],g(x_1)\in\C[x_1]$ we have
		$$(\psi_mx-(-1)^{|x|-1}x\psi_m)\cdot g(x_1)=0=[\psi_m,x]\cdot g(x_1),$$
		$$(\psi_m\psi_n+\psi_n\psi_m)\cdot f(x_0)=0=\delta_{m,-n}z\cdot f(x_0)=[\psi_m,\psi_n]\cdot f(x_0).$$
To prove that $\Omega_{\S(\delta)}(\lambda,a,b)$ is an $\S(\delta)$-module, it suffice to verify that
$$(L_m\psi_n-\psi_nL_m)\cdot f(x_0)=[L_m,\psi_n]\cdot f(x_0),\forall m\in\Z,n\in\delta+\Z,f(x_0)\in\C[x_0].$$
Similarly, to prove that $\Omega'_{\S(\delta)}(\lambda,a,b),\Omega''_{\S(\delta)}(\lambda,b)$ and $\tilde\Omega''_{\S(\delta)}(\lambda,b)$ are $\S(\delta)$-modules, it suffices to verify respectively that
$$(L_m\psi_n-\psi_nL_m)\cdot g(y_0)=[L_m,\psi_n]\cdot g(y_0),\forall m\in\Z,n\in\delta+\Z,g(y_0)\in\C[y_0],$$
$$(L_m\psi_n-\psi_nL_m)\cdot h(z_0)=[L_m,\psi_n]\cdot h(z_0),\forall m\in\Z,n\in\delta+\Z,h(z_0)\in\C[z_0],$$
$$(L_m\psi_n-\psi_nL_m)\cdot \tilde h(\tilde z_0)=[L_m,\psi_n]\cdot \tilde h(\tilde z_0),\forall m\in\Z,n\in\delta+\Z,\tilde h(\tilde z_0)\in\C[\tilde z_0].$$

Let $m\in\Z,n\in\delta+\Z$ and $f(x_0)\in\C[x_0],g(y_0)\in\C[y_0],h(z_0)\in\C[z_0],\tilde h(\tilde z_0)\in\C[\tilde z_0]$. Then we have
		\begin{equation*}\begin{split}
				(L_m\psi_n&-\psi_nL_m)\cdot f(x_0)\\
				=&L_m\cdot b\lambda^nf(x_1+n)-\psi_n\cdot\lambda^m(x_0+ma)f(x_0+m)\\
				=&b\lambda^{mn}(x_1-\frac{m}{2}+ma)f(x_1+m+n)-b\lambda^{m+n}(x_1+n+ma)f(x_1+n+m)\\
				=&b\lambda^{m+n}(-n-\frac{m}{2})f(x_1+m+n)\\
				=&[L_m,\psi_n]\cdot f(x_0),\\
                 (L_m\psi_n&-\psi_nL_m)\cdot g(y_0)\\
                =&L_m\cdot b\lambda^n(y_1-2na+3n)g(y_1+n)-\psi_n\cdot\lambda^m(y_0+ma)g(y_0+m)\\
                =&b\lambda^{m+n}(y_1-\frac{3m}{2}+ma)(y_1+m-2na+3n)g(y_1+m+n)\\
                &-b\lambda^{m+n}(y_1-2na+3n)(y_1+n+ma)g(y_1+n+m)\\
                =&(-n-\frac m2)b\lambda^{m+n}(y_1-(m+n)(2a-3))g(y_1+m+n)\\
                =&[L_m,\psi_n]\cdot g(y_0),\\
                 (L_m\psi_n&-\psi_nL_m)\cdot h(z_0)\\
                =&L_m\cdot b\lambda^n(z_1+3n)(z_1+n)h(z_1+n)-\psi_n\cdot\lambda^m(z_0+m)h(z_0+m)\\
                =&b\lambda^{m+n}(z_1-\frac 32m)(z_1+m+3n)(z_1+m+n)h(z_1+m+n)\\
                &-b\lambda^{m+n}(z_1+3n)(z_1+n)(z_1+m+n)h(z_1+m+n)\\
                =&(-n-\frac m2)b\lambda^{m+n}(z_1+3m+3n)(z_1+m+n)h(z_1+m+n)\\
                =&[L_m,\psi_n]\cdot h(z_0),\\
                 (L_m\psi_n&-\psi_nL_m)\cdot \tilde h(\tilde z_0)\\
                =&L_m\cdot b\lambda^n\tilde z_1(\tilde z_1-2n)\tilde h(\tilde z_1+n)-\psi_n\cdot \lambda^m(\tilde z_0+\frac 52m)\tilde h(\tilde z_0+m)\\
                =&b\lambda^{m+n}\tilde z_1(\tilde z_1+m)(\tilde z_1-2n+m)\tilde h(\tilde z_1+m+n)\\
                &-b\lambda^{m+n}\tilde z_1(\tilde z_1-2n)(\tilde z_1+n+\frac 52m)\tilde h(\tilde z_1+m+n)\\
                =&(-n-\frac m2)b\lambda^{m+n}\tilde z_1(\tilde z_1-2m-2n)\tilde h(\tilde z_1+m+n)\\
                =&[L_m,\psi_n]\cdot \tilde h(\tilde z_0).
		\end{split}\end{equation*}
So $\Omega_{\S(\delta)}(\lambda,a,b),\Omega'_{\S(\delta)}(\lambda,a,b),\Omega''_{\S(\delta)}(\lambda,b)$ and $\tilde\Omega''_{\S(\delta)}(\lambda,b)$ are $\S(\delta )$-modules.
%
%
\end{proof}

Let $V$ be an $\S(\delta)$-module such that $V_{\bar 0}$ and $V_{\bar 1}$ are free $\C[L_0]$-modules of rank $1$ by restriction, and that $\S(\delta)_{\bar 1}V\neq 0$. Next, we will show that $V$ is isomorphic to $\Omega_{\S(\delta)}(\lambda,a,b),\Omega'_{\S(\delta)}(\lambda,a,b),\Omega''_{\S(\delta)}(\lambda,b)$ or $\tilde\Omega''_{\S(\delta)}(\lambda,b)$ up to a parity-change for some $\lambda,b\in\C^*$ and some $a\in\C$.

Let $\1_0$ be the $\C[L_0]$-basis of $V_{\bar 0}$, and let $\1_1$ be the $\C[L_0]$-basis of $V_{\bar 1}$. Then $V=\C[L_0]\1_0\oplus\C[L_0]\1_1$.

From Lemma \ref{Ola} we know that  $c\cdot V_{\bar 0}=c\cdot V_{\bar 1}=0$, and there are $\lambda,\lambda'\in\C^*$ and $a,a'\in\C$ such that
$$\aligned
	&L_m\cdot f(L_0)\1_0=\lambda^m(L_0+ma)f(L_0+m)\1_0,\\
	&L_m\cdot f(L_0)\1_1=\lambda'^m(L_0+ma')f(L_0+m)\1_1,\endaligned$$
	for all $m\in\Z,f(L_0)\in\C[L_0].$

For any $m\in\delta+\Z$, let $f_m,g_m\in\C[L_0]$ be such that
$$\psi_m\cdot \1_0=f_m(L_0)\1_1,\psi_m\cdot\1_1=g_m(L_0)\1_0.$$
Using the fact that
$\psi_m  f(L_0)= f(L_0+m)\psi_m$ in $U(\S(\delta))$, we have the following formulas.
\begin{lemma}\label{f_m12} For any $m\in\delta+\Z,f(L_0)\in\C[L_0]$, we have
$$\aligned &\psi_m\cdot f(L_0)\1_0=f_m(L_0)f(L_0+m)\1_1,\\
&\psi_m\cdot f(L_0)\1_1=g_m(L_0)f(L_0+m)\1_0.\endaligned$$
\end{lemma}

\begin{lemma}\label{g_m=0}
If $g_r(L_0)=0$ (resp. $f_r(L_0)=0$) for some $r\in\delta+\Z$, then $g_m(L_0)=0$ (resp. $f_m(L_0)=0$) for all $m\in\delta+\Z$.
\end{lemma}
\begin{proof}
Suppose $g_r(L_0)=0$.
For any $n\in\Z$ we have
$$(-r-\frac n2)g_{r+n}(L_0)\1_0=(-r-\frac n2)\psi_{r+n}\cdot\1_1=[L_n,\psi_r]\cdot\1_1=L_n\psi_r\1_1-\psi_rL_n\1_1=0.$$
So $g_{r+n}(L_0)=0$ for all $n\in\Z\setminus\{-2r\}$.
Then
$$-\frac 12 g_{-r}(L_0)\1_0=-\frac 12\psi_{-r}\cdot\1_1=[L_{-2r-1},\psi_{r+1}]\cdot \1_1=0.$$
Therefore, $g_{-r}(L_0)=0$. So $g_m(L_0)=0$ for all $m\in\delta+\Z$.

Similarly, if $f_{r}(L_0)=0$, then $f_m(L_0)=0$ for all $m\in\delta+\Z$.
\end{proof}

By Lemma \ref{f_m12}, for any $m\in\delta+\Z$ we have
$$0=\psi_m^2\cdot \1_0=g_m(L_0)f_m(L_0+m)\1_0.$$
So $g_m(L_0)f_m(L_0+m)=0$.
Since the polynomial algebra $\C[L_0]$ has no nonzero zero divisors, we have $g_m(L_0)=0$ or $f_m(L_0+m)=0$.
Without loss of generality, assume that $g_m(L_0)=0$. Then $f_m(L_0)\neq 0$. In fact, if $f_m(L_0)=0$, then $f_r(L_0)=g_r(L_0)=0$ for all $r\in\delta+\Z$ by Lemma \ref{g_m=0}. This contradicts with $\S(\delta)_{\bar 1}V\neq 0$.

For any $m\in\delta+\Z$, let $d_m=\deg(f_m)$.
\begin{lemma} {\rm (1).}  $a'-a=-d_m-\frac{1}{2}$ for any $m\in\delta+\Z$;
	{\rm (2).}  $\lambda'=\lambda.$
\end{lemma}

\begin{proof}
 (1). For any $m\in\delta+\Z$ we have
\begin{equation*}\begin{split}
0=&[L_{-2m},\psi_m]\cdot\1_0\\
=&L_{-2m}\psi_m\cdot\1_0-\psi_mL_{-2m}\cdot\1_0\\
=&L_{-2m}\cdot f_m(L_0)\1_1-\psi_m\cdot \lambda^{-2m}(L_0-2ma)\1_0\\
=&\lambda'^{-2m}(L_0-2ma')f_m(L_0-2m)\1_1-\lambda^{-2m}(L_0+m-2ma)f_m(L_0)\1_1.
\end{split}\end{equation*}
So
\begin{equation}\label{eq12}
\lambda'^{-2m}(L_0-2ma')f_m(L_0-2m)=\lambda^{-2m}(L_0+m-2ma)f_m(L_0)\in\C[L_0].
\end{equation}
Comparing the coefficients of $L_0^{d_m+1}$ in (\ref{eq12}), we get $\lambda'^{-2m}=\lambda^{-2m}$ for any $m\in\delta+\Z$. So $$\lambda'^2=\lambda^2\text{ if }\delta=0, \text{ and }\lambda' =\lambda \text{ if }\delta= \frac12.$$

Let $f_m(L_0)=\mu(L_0-x_1)\dots(L_0-x_{d_m})$ for some $\mu\in\C^*$ and some $x_1,\dots,x_{d_m}\in\C$.
From (\ref{eq12}) we have
$$(L_0-2ma')(L_0-2m-x_1)\dots(L_0-2m-x_{d_m})=(L_0+m-2ma)(L_0-x_1)\dots(L_0-x_{d_m}).$$
Then
$2ma'+(2m+x_1)+\dots+(2m+x_{d_m})=(2ma-m)+x_1+\dots+x_{d_m},$
i. e., $$a'-a=-d_m-\frac{1}{2},\forall m\in\delta+\Z.$$

(2).
If $\delta=\frac12$ we already have  $\lambda'=\lambda$.
Now suppose $\delta=0$. By (1), $d_1=d_0$. Note that
\begin{equation*}\begin{split}
-\frac 12 f_1(L_0)\1_1=&[L_1,\psi_0]\cdot\1_0\\
=&L_1\cdot f_0(L_0)\1_1-\psi_0\cdot\lambda(L_0+a)\1_0\\
=&\lambda'(L_0+a')f_0(L_0+1)\1_1-\lambda(L_0+a)f_0(L_0)\1_1.
\end{split}\end{equation*}
Comparing the coefficients of $L_0^{d_0+1}$ we get $\lambda'=\lambda$.
\end{proof}

\begin{lemma}\label{cpt6.2}
Suppose $m\in\delta+\Z$ and $n\in\Z$. Then
$$(-m-\frac{n}{2})f_{m+n}(L_0)=\lambda^n(L_0+na')f_m(L_0+n)-\lambda^n(L_0+m+na)f_m(L_0).$$
\end{lemma}
\begin{proof}
This follows from the following computation:
\begin{equation*}\begin{split}
(-m-\frac{n}{2})&f_{m+n}(L_0)\1_1=(-m-\frac{n}{2})\psi_{m+n}\cdot\1_0\\
=&L_n\psi_m\cdot\1_0-\psi_m L_n\cdot\1_0\\
=&L_n\cdot f_m(L_0)\1_1-\psi_m\cdot\lambda^n(L_0+na)\1_0\\
=&\lambda^n(L_0+na')f_m(L_0+n)\1_1-\lambda^n(L_0+m+na)f_m(L_0)\1_1.
\end{split}\end{equation*}
So
$$(-m-\frac{n}{2})f_{m+n}(L_0)=\lambda^n(L_0+na')f_m(L_0+n)-\lambda^n(L_0+m+na)f_m(L_0).$$
\end{proof}

Write $d=d_m$ for any $m\in\delta+\Z$.
\begin{lemma}\label{roots}
Suppose $d>0$. For any $m\in\delta+\Z$ with $m>0$, the polynomial $f_m(x)$ has roots $x_1,\dots,x_d\in\C$ such that $x_{i+1}-x_i=2m,i=1,\dots,d-1$, and that $x_1=2ma',x_d=2ma-3m$.
\end{lemma}
\begin{proof}
Let $x_1,\dots,x_d$ be the roots of $f_m(x)$ with the positive part of $x_i$ being smaller or equal to $x_{i+1}$, $i=1,\dots,d-1$. From (\ref{eq12}) we have
$$(L_0-2ma')f_m(L_0-2m)=(L_0+m-2ma)f_m(L_0).$$
Then we have
$$\{2ma',x_1+2m,\dots,x_d+2m\}=\{2ma-m,x_1,\dots,x_d\}.$$
It is easy to see that $x_1\notin\{x_1+2m,\dots,x_d+2m\}$ and that $x_d+2m\notin\{x_1,\dots,x_d\}$.
So $x_1=2ma'$ and $x_d+2m=2ma-m$. By comparing the real parts of elements in these two sets, we have
$$x_1+2m=x_2,\dots,x_{d-1}+2m=x_d.$$
\end{proof}

\begin{lemma}
$d\leqslant 2$.
\end{lemma}
\begin{proof}
Let $r\in\delta+\Z$ with $r>0$.

Suppose $d\geqslant 3$.
By Lemma \ref{roots}, $f_r(x)$ has roots $x_1,\dots,x_d$ such that $x_{i+1}-x_i=2r,i=1,\dots,d-1$.
Then $f_{r}(x+1)$ has roots $x_1-2r,x_1,\dots,x_{d-1}$.

Let $m=r$ and $n=1$ in Lemma \ref{cpt6.2} and we get
\begin{equation}\label{eq3/2}
(-r-\frac 12)f_{r+1}(L_0)=\lambda(L_0+a')f_r(L_0+1)-\lambda(L_0+r+a)f_r(L_0).
\end{equation}
Then $x_1,x_2$ are roots of $f_{r+1}(x)$. By Lemma \ref{roots},
$$2r=x_2-x_1\in (2r+2)\Z,$$
which  is a contradiction. So $d\leqslant 2$.
\end{proof}

Fix some $\sqrt\lambda\in\C^*$ with $(\sqrt\lambda)^2=\lambda$, and set $\lambda^r=\lambda^{r-\frac{1}{2}}\sqrt\lambda$ for any $r\in\frac 12+\Z$.

\begin{lemma}
Suppose $d=2$. Then $a=1$ or $\frac 52$, and $V$ is isomorphic to $\Omega''_{\S(\delta)}(\lambda,b)$ or $\tilde\Omega''_{\S(\delta)}(\lambda,b)$   for some $b\in\C^*$.
\end{lemma}
\begin{proof}
Let $r\in\delta+\Z$ with $r>0$. Suppose $d=2$. By Lemma \ref{roots}, the roots of $f_r(x)$ are $2ra'$ and $2ra-3r$. Then
$$f_r(x)=\mu(x-2ra')(x-2ra+3r)$$
for some $\mu\in\C^*$.
From (\ref{eq3/2}) we have
\begin{equation}\label{eqfr+1}\begin{split}
(-r-\frac 12)f_{r+1}(L_0)=&\mu\lambda(L_0+a')(L_0+1-2ra')(L_0+1-2ra+3r)\\
&-\mu\lambda(L_0+r+a)(L_0-2ra')(L_0-2ra+3r).
\end{split}\end{equation}
By Lemma \ref{roots}, the roots of $f_{r+1}(x)$ are $(r+1)(2a-5)$ and $(r+1)(2a-3)$. Then
$$0=(-r-\frac 12)f_{r+1}((r+1)(2a-5))=\mu\lambda(a-1)(2a-5)(2r+1).$$
So $a=1$ or $\frac 52$.

Let $b=\mu\lambda^{-r}$. Then
$$f_r(x)=b\lambda^r(x-2ra')(x-2ra+3r).$$
From Lemma \ref{cpt6.2} we have
$$f_m(L_0)=b\lambda^m(L_0+3m)(L_0+m),\forall m\in\delta+\Z$$
if $a=1$, and
$$f_m(L_0)=b\lambda^mL_0(L_0-2m),\forall m\in\delta+\Z$$
if $a=\frac 52$. By the definition of $\Omega''_{\S(\delta)}(\lambda,b)$ and $\tilde\Omega''_{\S(\delta)}(\lambda,b)$, $V$ is isomorphic to $\Omega''_{\S(\delta)}(\lambda,b)$ or $\tilde\Omega''_{\S(\delta)}(\lambda,b)$ as $\S(\delta)$-modules.
\end{proof}

\begin{lemma}
Suppose $d=1$. Then $V$ is isomorphic to $\Omega'_{\S(\delta)}(\lambda,a,b)$ for some $b\in\C^*$.
\end{lemma}
\begin{proof}
Let $r\in\delta+\Z$ with $r>0$. By Lemma \ref{roots}, the root of $f_r(x)$ is $2ra-3r$. Then
$$f_r(L_0)=\mu(L_0-2ra+3r)$$
for some $\mu\in\C^*$.

By Lemma \ref{cpt6.2}, for any $n\in\Z$ we have
\begin{equation*}\begin{split}
(-r-\frac n2)f_{n+r}(L_0)=&\lambda^n(L_0+na')f_r(L_0+n)-\lambda^n(L_0+na+r)f_r(L_0)\\
=&\lambda^n\mu(L_0+na')(L_0-2ra+3r+n)\\
&-\lambda^n\mu(L_0+na+r)(L_0-2ra+3r)\\
=&\lambda^n\mu(-r-\frac n2)(L_0-(2a-3)(n+r)).
\end{split}\end{equation*}
So
$$f_{n+r}(L_0)=\lambda^n\mu(L_0-(2a-3)(n+r))$$
if $n\neq -2r$. In particular,
$$f_{3r}(L_0)=\lambda^{2r}\mu(L_0-3(2a-3)r).$$
Let $m=3r,n=-4r$ in Lemma \ref{cpt6.2} and we get
\begin{equation*}\begin{split}
-r f_{-r}(L_0)=&\lambda^{-4r}(L_0-4ra')f_{3r}(L_0-4r)-\lambda^{-4r}(L_0+3r-4ra)f_{3r}(L_0)\\
=&-r\mu\lambda^{-2r}(L_0+2ra-3r).
\end{split}\end{equation*}
Hence,
$$f_{n+r}(L_0)=\lambda^n\mu(L_0-2(n+r)a+3(n+r)),\forall n\in\Z.$$
Then
$$f_m(L_0)=b\lambda^m(L_0-2ma+3m),\forall m\in\delta+\Z,$$
where $b=\mu\lambda^{-r}$.
By the definition of $\Omega'_{\S(\delta)}(\lambda,a,b)$, $V$ is isomorphic to $\Omega'_{\S(\delta)}(\lambda,a,b)$ as $\S(\delta)$-modules.
\end{proof}

\begin{lemma}
Suppose $d=0$. Then $V$ is isomorphic to $\Omega_{\S(\delta)}(\lambda,a,b)$ for some $b\in\C^*$.
\end{lemma}
\begin{proof}
Let $r\in\delta+\Z\setminus\{0\}$. Then $f_r(L_0)=b\lambda^r$ for some $b\in\C^*$.
By Lemma \ref{cpt6.2}, for any $n\in\Z$ we have
$$(-r-\frac{n}{2})f_{r+n}(L_0)=b\lambda^{n+r}(L_0+na')-b\lambda^{n+r}(L_0+r+na)=b\lambda^{n+r}(-r-\frac{n}{2}).$$
So $f_{r+n}(L_0)=b\lambda^{n+r}$ for any $n\neq -2r$.
In particular, $f_{3r}(L_0)=b\lambda^{3r}$. Let $m=3r,n=-4r$ in Lemma \ref{cpt6.2} and we get
$$-r f_{-r}(L_0)=\lambda^{-4r}(L_0-4ra')f_{3r}(L_0-4r)-\lambda^{-4r}(L_0+3r-4ra)f_{3r}(L_0)=-rb\lambda^{-r}.$$
Hence
$$f_{m}(L_0)=b\lambda^{m},\forall m\in\delta+\Z.$$
By the definition of $\Omega_{\S(\delta)}(\lambda,a,b)$, $V$ is isomorphic to $\Omega_{\S(\delta)}(\lambda,a,b)$ as $\S(\delta)$-modules.
\end{proof}

From the discussion above, we have
\begin{theorem}\label{key}
Suppose $V$ is an $\S(\delta)$-module such that $V_{\bar 0}$ and $V_{\bar 1}$ are free $\C[L_0]$-modules of rank $1$ by restriction. If $\S(\delta)_{\bar 1}V\neq 0$, then $V$ is isomorphic to $\Omega_{\S(\delta)}(\lambda,a,b)$, $\Omega'_{\S(\delta)}(\lambda,a,b),\Omega''_{\S(\delta)}(\lambda,b)$ or $\tilde\Omega''_{\S(\delta)}(\lambda,b)$ up to a parity-change for some $\lambda,b\in\C^*$ and some $a\in\C$.
\end{theorem}

It is easy to see that all these $\S(\delta)$-modules are not simple since $V_{\bar 1}$ is a $\S(\delta)$-submodule of $V$.

\


We conclude this paper with the following questions.
\begin{itemize}
	\item Is there any way to construct new   simple smooth $\S(\delta)$-modules of nonzero level? Or  can we  classify such $\S(\delta)$-modules?
	\item Is there any way to construct new    simple $F(\delta)$-modules of nonzero level? Or  can we  classify such $\S(\delta)$-modules?
\end{itemize}

\

\noindent {\bf Acknowledgement.} { The authors would like to thank Rencai L\"u and Yao Ma for   helpful discussions. Part of the research in this paper was carried out during the visit of the first
	author at Wilfrid Laurier University in 2022.
	The first author is partially supported by CSC of China (No. 202006920055), and National Natural Science Foundation of China (Grants No. 11971440).
	The second author is partially supported by  National Natural Science Foundation of China (Grant No. 11871190) and NSERC (311907-2020).
}

\

Y. Xue: School of Sciences, Nantong University, Nantong, Jiangsu, 226019, P. R. China.

Email address: yxue@ntu.edu.cn

\
K. Zhao: Department of Mathematics, Wilfrid Laurier University, Waterloo, ON, Canada N2L 3C5, and School of Mathematical Science, Hebei Normal (Teachers) University, Shijiazhuang, Hebei, 050024 P. R. China.

Email address: kzhao@wlu.ca

\end{document}